\documentclass[10pt]{article} 
\usepackage{multicol} 
\usepackage{newlfont}
\usepackage{amsmath,amssymb}
\usepackage{tikz}

\def\rit{\mathbb{R}}
\def\zit{\mathbb{Z}}   
\def\nit{\mathbb{N}}

\def\ppit{\mathbb{P}} 
\def\qit{\mathbb{Q}} 
\def\cit{\mathbb{C}}

\newcommand{\pf}{{\em Proof.~}}
\newcommand{\qed}{\hfill~~\mbox{$\Box$}}

\newenvironment{proof}{\smallskip \noindent \pf}{\qed \bigskip}

\newtheorem{theorem}{Theorem}[section]
\newtheorem{proposition}[theorem]{Proposition}
\newtheorem{definition}[theorem]{Definition}

\newtheorem{corollary}[theorem]{Corollary}

\newtheorem{remark}[theorem]{Remark}
\newtheorem{example}[theorem]{Example}

\DeclareMathOperator{\card}{Card}

\DeclareMathOperator{\age}{age}
\DeclareMathOperator{\vol}{vol}
\DeclareMathOperator{\boite}{Box}

\DeclareMathOperator{\gr}{gr}
\DeclareMathOperator{\Spec}{Spec}
\DeclareMathOperator{\conv}{conv}

\DeclareMathOperator{\orb}{orb}

\DeclareMathOperator{\supp}{supp}

\begin{document}


\title{\bf Hard Lefschetz properties and distribution of spectra in singularity theory and Ehrhart theory}
\author{\sc Antoine Douai\thanks{Mathematics Subject Classification 52B20, 32S40, 14J33.
}\\ 
Universit\'e C\^ote d'Azur, CNRS, LJAD, FRANCE\\
Email address: antoine.douai@univ-cotedazur.fr}

\maketitle

\begin{abstract}
Motivated by the distribution of spectra in singularity theory and combinatorics, we study a hard Lefschetz property for Laurent polynomials and for polytopes
and we give combinatorial criteria for this property to be true. 
This also provides informations about a conjecture by Katzarkov-Kontsevich-Pantev. 
\end{abstract}

\section{Introduction}

Let $P$ be a lattice polytope in $\rit^n$ (the convex hull of a finite set in $N:=\zit^n$).
Define, for a positive integer $\ell$, $L_P (\ell ):= \card ( (\ell P )\cap N)$. Then $L_P$ is a polynomial in $\ell$ of degree $n$, the {\em Ehrhart polynomial} of $P$ and 
\begin{equation}\label{eq:serie Ehrhart}
1+\sum_{m\geq 1}L_P (m) z^m =\frac{\delta_0 +\delta_1 z +\cdots +\delta_n z^n}{(1-z)^{n+1}}
\end{equation}
where the $\delta_j$'s are nonnegative integers. The vector
$\delta_P =( \delta_0 ,\cdots ,\delta_n )$ is called the {\em $\delta$-vector} of the polytope $P$. The first result in the study of the  distribution of the $\delta$-vector is probably Hibi's symmetry property $\delta_i =\delta_{n-i}$ for $i=1,\cdots ,n$, which is actually a characterization of reflexive polytopes \cite{Hibi}. 
The second one concerns the unimodality of the $\delta$-vector of a reflexive polytope: taking into account the previous symmetry property, one could expect 
$\delta_0 \leq \delta_1 \leq \cdots \leq \delta_{[n/2]}$ and $\delta_{[n/2]}\geq \delta_{[n/2]+1}\geq \cdots \geq \delta_n$. This is indeed what happens in dimension less than or equal to five \cite{Hibi0}, but this unimodality may fail in dimension greater than or equal to six, see for instance \cite[Example 3.4]{Sta}, \cite{MP}, \cite{Payne}. 

On the other hand, singularity theory meets Ehrhart theory by the means of the $\delta$-vector: the spectrum at infinity of a tame Laurent polynomial determines the $\delta$-vector of its Newton polytope and both coincide if the latter is reflexive \cite{D12}. This interplay encourages us also to study the unimodality (and more generally, the distribution) of the spectrum at infinity of a regular function. 

Classically, unimodality can be seen as a combinatorial application of the hard Lefschetz theorem (see \cite{Stanley} for instance where it is shown that 
the Poincar\'e polynomial of a smooth complex projective variety is unimodal) and 
we are naturally led to study a hard Lefschetz property for regular functions (singularity side) and for polytopes (Ehrhart theory side). 
On the singularity side, the hard Lefschetz property for a Laurent polynomial $f$ is provided by the 
multiplication by $f$ on a graded Jacobi ring. The hard Lefschetz property for a simplicial polytope $P$ is provided by the hard Lefschetz property for the orbifold cohomology of the orbifold associated with $P$ by the work of Borisov, Chen and Smith \cite{BCS}.
Both are related by a mirror theorem.
This is detailed in Section \ref{sec:HardLefschetzLaurentSimplices},
where we  also give  a combinatorial criterion
for these hard Lefschetz properties to be satisfied:
let $P$ be a full dimensional, simplicial, lattice polytope in $\rit^n$ containing the origin as an interior point and let $\Sigma_P$ be the fan over the proper faces of $P$. For a $n$-dimensional cone $\sigma \in\Sigma_P$, let $\boite (\sigma )$ be the set of $v\in N$ such that $v=\sum_{\rho_i \subseteq\sigma} q_i b_i $ for some $0\leq q_i <1$, where $\rho_i$ denotes the ray generated by the vertex $b_i$ of $P$. 
We have the following generalization of \cite[Proposition 4.1]{Fe} (see Proposition \ref{prop:ConcretAge}):

\begin{theorem} 
\label{theo:Intro}
The polytope $P$ satisfies the hard Lefschetz property 
if and only if
$$[\nu (v) ]=(\dim \sigma (v) -1)/2\ \mbox{if}\ \nu (v)\notin\nit$$
and 
$$\nu (v)=\dim \sigma (v)/2 \ \mbox{if}\ \nu (v)\in\nit$$
for all $v\in\cup_{\sigma}\boite (\sigma )$ (the union is taken over all the $n$-dimensional cones of $\Sigma_P $), where
$\sigma (v)$ denotes the smallest cone of $\Sigma_P$ containing $v$ and $\nu$ is the Newton function  of $P$.
\end{theorem}

When applied to a reduced simplex $\Delta$, this criterion reduces to an arithmetic condition on its
weight, see Proposition \ref{prop:CNSHLPolSimplex} (the {\em weight} of a simplex $\Delta :=\conv (v_0 ,\cdots ,v_n )$ is
the tuple $(q_0 ,\cdots , q_n )$, arranged by increasing order, where 
$q_i := |\det (v_0 ,\cdots , \widehat{v_i},\cdots , v_n )|$ and 
the simplex $\Delta$ is {\em reduced} if $\gcd (q_0 ,\cdots ,q_n )=1$, see Section \ref{sec:simplices} for details).
For instance, if moreover $\Delta$ is 
{\em reflexive}, that is if $q_i$ divides $\mu :=q_0 +\cdots +q_n$ for $i=0,\cdots ,n$, we get:

\begin{proposition}\label{prop:NCHLSimplicesIntro}
Assume that the reduced and reflexive simplex $\Delta$ of weight $(q_0 ,\cdots ,q_n )$ satisfies the hard Lefschetz property. Then, 
\begin{equation}\label{eq:NCHLSimplicesIntro}
\frac{2\mu }{q_n }=n+1 +m(q_n ) 
\end{equation}
where $m(q_n )$ denotes the multiplicity of $q_n$ in the ordered tuple $(q_0 ,\cdots , q_n )$.
\end{proposition}
See Remark \ref{rem:NCHLSimplices}. For example, (\ref{eq:NCHLSimplicesIntro}) fails for
the three dimensional reflexive  and reduced simplex $\Delta$ of weight $(1,1,1,3)$: this simplex does not satisfy the hard Lefschetz property. 
Actually, we have a stronger statement (a necessary and sufficient condition, see Corollary \ref{coro:NCSHLSimplices}) and
it follows from our computations that the hard Lefschetz properties are not common at all (and this answers a question in \cite[Section 4]{Fe}): for instance, we check that the hard Lefschetz property is true for 2 out the 14 three dimensional reduced and reflexive simplices described in \cite{Conrads}. 

This has an interpretation in Hodge theory: it has been noticed in \cite{Sab} that a Laurent polynomial $f$ satisfies the hard Lefschetz property if and only if the mixed Hodge structure produced by the Laplace transform of its Gauss-Manin system is of Hodge-Tate type. As a consequence, we get
informations about a conjecture by Katzarkov-Kontsevich-Pantev \cite[Conjecture 3.6]{KKP}, see Proposition \ref{prop:KKPSimplices}: it turns out that only a few Laurent polynomials whose Newton polytopes are reduced and reflexive simplices satisfy this conjecture.

Last, and this was after all our starting point, the hard Lefschetz properties studied in this paper are related with the unimodality of the spectrum at infinity of a regular function. 
This is discussed in Section \ref{sec:Distribution}. 

These notes were motivated by Sabbah's paper \cite{Sab}, in which the "smooth" case is considered.

\section{Spectra}

\label{sec:Spectra}

In this section, we recall some results from \cite{D12}. Let $N$ be the lattice $\zit^n$ and
let $P\subset N_{\rit}$ be a full dimensional lattice polytope containing the origin as an interior point. 
We assume throughout this paper that $P$ is {\em simplicial}. 

Let $\Sigma_P$ be the (simplicial) fan in $N_{\rit}$ obtained by taking the cones over the proper faces of $P$ and let $X_{\Sigma_P}$ be the complete, projective, toric variety  of the fan $\Sigma_P$.
The {\em Newton function} of $P$ is the function
$\nu :N_{\rit}\rightarrow \rit$
which takes the value $1$ at the vertices of $P$ and which is linear on each cone of $\Sigma_P$. 
The {\em Milnor number} of $P$ is $\mu_P :=n! \vol (P)$
where the volume $\vol (P)$ is normalized such that the volume of the cube is equal to $1$. 
We define the {\em Newton spectrum of} $P$ by
\begin{equation}\label{eq:DefSpecP}
\Spec_{P}(z):=(1-z)^n \sum_{v\in N} z^{\nu (v)}.
\end{equation}

Let $f(u)=\sum_{m\in\zit^n} a_m u^{m}$ be a Laurent polynomial defined on $(\cit^* )^n$. 
The {\em Newton polytope} $P$ of $f$ is the convex hull of $\supp f :=\{m\in\zit^n , \ a_m \neq 0\}$ in $\rit^n$. 
We assume in this text that $f$ is convenient (its Newton polytope contains the origin as an interior point) and nondegenerate in the sense of Kouchnirenko \cite{K}. 
Let $\mathcal{A}_f :=\mathcal{B}/ \mathcal{L}$ where 
$\mathcal{B}:=\cit [u_1 , u_1^{-1},\cdots , u_n ,u_n^{-1}]$
and $\mathcal{L}:=(u_1\frac{\partial f}{\partial u_1},\cdots , u_n\frac{\partial f}{\partial u_n})$
is the ideal generated by the partial derivative $u_1\frac{\partial f}{\partial u_1},\cdots , u_n\frac{\partial f}{\partial u_n}$ of $f$.
We define an increasing filtration $\mathcal{N}_{\bullet}$ on $\mathcal{B}$, indexed by $\qit$, by setting  
\begin{equation}\nonumber
\mathcal{N}_{\alpha}\mathcal{B}:=\{g\in \mathcal{B},\ \supp (g) \in \nu^{-1}(]-\infty ; \alpha ]) \}
\end{equation} 
where $\nu$ is the Newton function of the Newton polytope $P$ of $f$ and
$\supp (g)=\{m\in\nit^n,\ a_m \neq 0\}$ if $g=\sum_{m\in \nit^n}a_m u^m \in\mathcal{B}$. 
By projection, the Newton filtration $\mathcal{N}_{\bullet}$ on $\mathcal{B}$ induces
the Newton filtration $\mathcal{N}_{\bullet}$ on $\mathcal{A}_f$ and 
the {\em spectrum at infinity} of $f$ is given by
\begin{equation}\label{def:SpectrumInfinity}
\Spec_f (z)=\sum_{\alpha\in\qit}\dim_{\cit}(\gr_{\alpha} ^{\mathcal{N}}\mathcal{A}_f )   z^{\alpha}.
\end{equation}

\noindent Both spectra are related: if $f$ is a convenient and nondegenerate Laurent polynomial with Newton polytope  $P$, we have $\Spec_f (z)=\Spec_P (z)$, see \cite[Corollary 2.2]{D12}.

We are interested in the distribution of $\Spec_f (z)$ and $\Spec_P (z)$ and it will be useful to decide when these spectra are polynomials.
Recall that a lattice polytope $P$ is {\em reflexive} if it contains the origin as an interior point and
if its polar polytope $P^{\circ}:=\{y\in M_{\rit},\ \langle y, x \rangle\leq 1\ \mbox{for all}\ x\in P\}$ is a lattice polytope.

\begin{proposition}\label{prop:ReflexivevsIntegral}\cite[Proposition 5.1]{D12}
The following are equivalent:
\begin{enumerate}
\item $\Spec_P (z)$ is a polynomial, 
\item $P$ is reflexive,
\item $\Spec_{P}(z)=\delta_0 +\delta_1 z +\cdots +\delta_n z^n$ where $(\delta_0 ,\cdots , \delta_n )$ is the $\delta$-vector of $P$.\qed
\end{enumerate}
\end{proposition}

On the singularity side, we get (and we will refer to this case as the {\em unipotent} case):

\begin{corollary}\label{coro:UnipotentReflexif}
Let $f$ be a convenient and nondegenerate Laurent polynomial.
Then its spectrum at infinity $\Spec_f (z)$ is a polynomial if and only if its Newton polytope $P$ is reflexive.\qed
\end{corollary}

\section{The hard Lefschetz property for Laurent polynomials and polytopes}

\label{sec:HardLefschetzLaurentSimplices}

Let $f$ be a  convenient and nondegenerate Laurent polynomial defined on $(\cit^* )^n$. 
The multiplication by $f$ induces maps 
\begin{equation}\nonumber 
[f]:\gr_{\alpha} ^{\mathcal{N}}\mathcal{A}_f \longrightarrow \gr_{\alpha +1} ^{\mathcal{N}}\mathcal{A}_f 
\end{equation}
for $\alpha\in\qit$. The following definition can already be found in \cite{Sab}:

\begin{definition} 
\label{def:HLforfGen}
Let $f$ be a convenient and nondegenerate Laurent polynomial on $(\cit^* )^n$. We will say that $f$ satisfies the hard Lefschetz property (HL) if
the multiplication by $f$ induces isomorphisms
\begin{equation}\label{eq:HLforfGen1}
[f]^{n-1-2k}:\gr_{\alpha +k}^{\mathcal{N}}\mathcal{A}_f \stackrel{\cong}{\longrightarrow} \gr_{\alpha +n-1-k} ^{\mathcal{N}}\mathcal{A}_f 
\end{equation}
for $0\leq k\leq [(n-1)/2]$ and $\alpha\in ]0,1[$ and
\begin{equation}\label{eq:HLforfGen2}
[f]^{n-2k}:\gr_k ^{\mathcal{N}}\mathcal{A}_f \stackrel{\cong}{\longrightarrow} \gr_{n-k} ^{\mathcal{N}}\mathcal{A}_f 
\end{equation}
for $0\leq k\leq [n/2]$.
\end{definition}

Let now $P$ be a simplicial full-dimensional lattice polytope in $\rit^n$ containing the origin as an interior point. We will denote by $\mathcal{V}(P ):=\{b_1 ,\cdots ,b_r\}$ the set of its vertices.
Let
\begin{itemize}
\item $\mathcal{X}$ be the Deligne-Mumford stack associated with the stacky fan $\mathbf{\Sigma} :=(\zit^n, \Sigma_P , \mathcal{V}(P))$ by \cite[Section 3]{BCS}, 
\item $I_{\mathcal{X}}=\coprod_{\ell\in F} \mathcal{X_{\ell}}$ be the decomposition into connected components of the inertia orbifold of $\mathcal{X}$, see \cite[Section 4.1]{CR}, 
\item $H_{\orb}^{2\alpha}(\mathcal{X}, \cit ):=\oplus_{\ell\in F} H^{2(\alpha-\age (\mathcal{X}_{\ell}))}(\mathcal{X}_{\ell}, \cit )$ be the orbifold cohomology groups of $\mathcal{X}$,
where $\age (\mathcal{X}_{\ell})$ the age of the sector $\mathcal{X}_{\ell}$, see \cite[Definition 4.8]{CR},
\item $f_P$ be the Laurent polynomial on $(\cit^*)^n$ defined by
$f_P (u):=\sum_{b\in\mathcal{V}(P)} u^b $.
\end{itemize}

\noindent See \cite[Proposition 4.7]{BCS} 
for a toric description of the sectors $\mathcal{X}_{\ell}$. 

 The following wonderful result is due to \cite{BCS}, with a little help from \cite{K} (the orbifold cohomology is equipped with the orbifold cup-product $\cup_{\orb}$, see \cite[Section 6]{BCS}).

\begin{proposition} \label{prop:Mir} \cite{BCS}
There is an isomorphism of $\qit$-graded rings
\begin{equation}\nonumber 
\varphi : H_{\orb}^{2*}(\mathcal{X} , \cit )\stackrel{\cong}{\longrightarrow} \gr^{\mathcal{N}}_{*}\mathcal{A}_{f_P}.
\end{equation}
\end{proposition}
\begin{proof} Notice first that $f_P$ is convenient (because $P$ contains the origin as an interior point) and nondegenerate (thanks to the simpliciality assumption) with respect to its Newton polytope $P$. By \cite[Th\'eor\`eme 4.1]{K}, the map 
$\partial : \cit [u, u^{-1}]^n \rightarrow \cit [u, u^{-1}]$
defined
by $\partial (b_1 ,\cdots , b_n )=b_1 u_1 \frac{\partial f}{\partial u_1 }+\cdots + b_n u_n \frac{\partial f}{\partial u_n }$ is strict with respect 
to the Newton filtration. Hence, 
and by the definition of the Newton filtration,
the graded ring $\gr_{*}^{\mathcal{N}}\mathcal{A}_{f_P}$ is nothing but the "Stanley-Reisner presentation" of $\mathcal{X}$ given by the right hand side of \cite[Theorem 1.1]{BCS} and the result follows from {\em loc. cit.}
\end{proof}

\noindent It should be emphasized that Proposition \ref{prop:Mir} provides an isomorphism of {\em rings}, and this really depends on the special form of $f_P$, from which we also get $\varphi^{-1} ([f_P ])\in H^{2}(\mathcal{X}_{0}, \cit )$ where $\mathcal{X}_{0}$ denotes the untwisted sector.

 The cohomology $H^{*}(\mathcal{X}_{\ell},  \cit)$ of the twisted sector $\mathcal{X}_{\ell}$ is a $H^{*} (\mathcal{X}_{0}, \cit)$-module under the orbifold cup-product, and this module structure is basically given by the standard cup-product on $H^{*}(\mathcal{X}_{\ell}, \cit)$, see for instance \cite[Proposition 3.2]{Fe}, \cite[Proof of Theorem 1.1]{BCS}. We define, for 
$\omega\in H^{2} (\mathcal{X}_0 ,\cit )$, 
\begin{equation}
L_{\omega} (\eta ):=\pi_0^* \omega \cup_{\orb} \eta 
\end{equation}
where $\pi_0$ denotes the restriction of $\pi : I_{\mathcal{X}}\longrightarrow \mathcal{X}$ to the non-twisted sector.
We have the following counterpart of Definition \ref{def:HLforfGen}:

\begin{definition} 
\label{def:HLforPGen}
We will say that $P$ satisfies the hard Lefschetz property (HL) if there exists $\omega\in H^{2}(\mathcal{X}_{0}, \cit )$ such that the orbifold cup-product by $\omega$ induces isomorphisms
\begin{equation}\label{eq:LefConditionsPolytope1}
L_{\omega}^{n-1-2k} : H^{2(\alpha +k)}_{\orb} (\mathcal{X}, \cit )\stackrel{\cong}{\longrightarrow}H^{2(\alpha +n-1-k)}_{\orb}(\mathcal{X}, \cit )
\end{equation}  
for $0\leq k\leq [(n-1)/2]$ and $\alpha\in ]0,1[$ and
\begin{equation}\label{eq:LefConditionsPolytope2}
L_{\omega}^{n-2k} : H^{2k}_{\orb} (\mathcal{X}, \cit )\stackrel{\cong}{\longrightarrow}H^{2(n-k)}_{\orb}(\mathcal{X}, \cit ) 
\end{equation}
for $0\leq k\leq [n/2]$.
\end{definition}

\begin{remark} \label{rem:LefschetzWPS}
Assume that $P$ is a (reduced) simplex, as in Section \ref{sec:simplices} below (and we will mainly consider this situation).
The corresponding orbifold is a weighted projective space $\ppit (w)$.  We have $\dim H^2 (\mathcal{X}_{0}, \cit )=1$ and this vector space is generated by the Chern class $\omega :=c_1 (\mathcal{O}_{\ppit (w)} (1))$
(see \cite[Proposition 3.6 and Remark 3.7 ]{Man} for the definition of $\mathcal{O}_{\ppit (w)} (1)$ and its restrictions to the various sectors). The action of $L_{\omega}$ is then computed using Corollary 3.18 of loc. cit.
\end{remark}

We now give criteria for this hard Lefschetz property to be true. Let $P$ be a simplicial full dimensional lattice polytope in $\rit^n$ containing the origin as an interior point and let $\mathcal{X}$ as above.
For $\ell\in F$, we put $\ell^{-1}:=I(\ell )$ where $I$ is the involution on $F$ induced by the involution  
on the inertia orbifold $I_{\mathcal{X}}=\coprod_{\ell\in F} \mathcal{X_{\ell}}$ defined in \cite[(4.3)]{CR}. 
We will denote by $[x]$ the integral part of $x$. First, we have the following generalization of a result of Fernandez \cite{Fe}:

\begin{theorem}\label{theo:CNSHLPolytope}
The polytope $P$ satisfies (HL) if and only if
\begin{equation}\label{eq:CNSAgeNonEntier}
[\age (\mathcal{X}_{\ell})]=[\age (\mathcal{X}_{\ell^{-1}})] 
\end{equation}
for all $\ell\in F$.
\end{theorem}

\begin{proof} In what follows, we put $i_{\ell}:=\age (\mathcal{X}_{\ell})$.
Assume first that $P$ satisfies the hard Lefschetz property (\ref{eq:LefConditionsPolytope1}). Let $\alpha\in ]0,1[$.
Because the orbifold cup-product by $\omega\in H^{2}(\mathcal{X}_{0}, \cit )$ preserves the cohomology of each sector $\mathcal{X}_{\ell}$,
we get the isomorphisms
\begin{equation}\label{eq:Inter1}
L_{\omega}^{n-1-2k} : H^{2(\alpha +k-i_{\ell})} (\mathcal{X}_{\ell},\cit)\stackrel{\cong}{\longrightarrow}H^{2(\alpha +n-1-k-i_{\ell})} (\mathcal{X}_{\ell}, \cit).
\end{equation}
for $\ell\in F$ and $k\leq [(n-1)/2]$, where $L_{\omega}(\eta )= \pi^*_{\ell} \omega\wedge\eta \in H^{q+2} (\mathcal{X}_{\ell},\cit)$
if $\eta\in H^{q} (\mathcal{X}_{\ell},\cit)$ and
$\pi_{\ell}$ denotes the restriction of $\pi$ to $\mathcal{X}_{\ell}$.
Since (\ref{eq:Inter1}) is relevant only if $\alpha -i_{\ell}\in\zit$, we may assume that $\alpha =i_{\ell}-[i_{\ell}]$.

By \cite[Lemma 4.6]{CR}, we have $n_{\ell}:=\dim \mathcal{X}_{\ell}=n-i_{\ell}-i_{\ell^{-1}}$ and it follows that the isomorphisms (\ref{eq:Inter1}) are equivalent to 
\begin{equation}\nonumber
L_{\omega}^{n-1-2k} : H^{2(k-[i_{\ell}])} (\mathcal{X}_{\ell},\cit)\stackrel{\cong}{\longrightarrow}H^{2(n_{\ell}-1-k
+i_{\ell}+i_{\ell^{-1}} -[i_{\ell}])} (\mathcal{X}_{\ell},\cit).
\end{equation}
Because $i_{\ell}+i_{\ell^{-1}}\in\zit$ and $i_{\ell}\notin\zit$, we have $i_{\ell}+ i_{\ell^{-1}}=[i_{\ell}]+[i_{\ell^{-1}}]+1$  and we finally get the isomorphisms
 \begin{equation}\label{eq:Inter3}
L_{\omega}^{n-1-2k} : H^{2(k-[i_{\ell}])} (\mathcal{X}_{\ell},\cit)\stackrel{\cong}{\longrightarrow}H^{2(n_{\ell}-k
+[i_{\ell^{-1}}])} (\mathcal{X}_{\ell},\cit).
\end{equation}
Since $i_{\ell}+i_{\ell^{-1}}\leq n$, we may assume that
$[i_{\ell}]\leq [(n-1)/2]$ and we can put $k=[i_{\ell}]$ in (\ref{eq:Inter3}) in order to get the isomorphism
\begin{equation}\nonumber 
H^{0} (\mathcal{X}_{\ell}, \cit)\stackrel{\cong}{\longrightarrow}H^{2(n_{\ell}-[i_{\ell}]
+[i_{\ell^{-1}}])} (\mathcal{X}_{\ell},\cit).
\end{equation}
It follows that $[i_{\ell^{-1}}]-[i_{\ell}]\leq 0$. 
In particular, we have also $[i_{\ell^{-1}}]\leq [(n-1)/2]$ and,
by symmetry,  we get $[i_{\ell}]-[i_{\ell^{-1}}]\leq 0$. 
This shows that $[i_{\ell}]=[i_{\ell^{-1}}]$ if $P$ satisfies the hard Lefschetz property (\ref{eq:LefConditionsPolytope1}).
 The result is shown similarly if $P$ satisfies the hard Lefschetz property (\ref{eq:LefConditionsPolytope2}).

 We get the converse going backward, applying the hard Lefschetz theorem for the cohomology of $\mathcal{X}_{\ell}$, see for instance \cite[Theorem 12.5.8 and (12.5.2)]{CLS} where $\omega$ is the cohomology class of an ample divisor (by \cite[Proposition 4.7]{BCS}, its restrictions to the twisted sectors are also ample because a strictly convex function descends to a a strictly convex function on the quotient fan, see for instance \cite[p. 12]{Gross}).
\end{proof}

\begin{remark} Theorem \ref{theo:CNSHLPolytope} has been suggested by $\cite{Fe}$. If
$P$ is reflexive, the ages are integers and $P$ satisfies (HL) if and only if
$\age (\mathcal{X}_{\ell})=\age (\mathcal{X}_{\ell^{-1}})$. This result is already stated in loc. cit.
\end{remark}

\begin{remark} \label{rem:Smooth}
Assume that the toric variety $X_{\Sigma_P}$ is smooth: we have
$\mathcal{X}_P =X_{\Sigma_P}$ and equality (\ref{eq:CNSAgeNonEntier}) holds true since there are no twisted sectors. This matches with \cite[Proposition 3.4]{Sab}. 
\end{remark}

\begin{corollary}\label{coro:CNSHLPol}
If $f_P$ satisfies (HL) if and only if $[\age (\mathcal{X}_{\ell})]=[\age (\mathcal{X}_{\ell^{-1}})]$
for all $\ell\in F$.
\end{corollary}
\begin{proof} 
Assume first that $f_P$ satisfies (HL).
We use Proposition \ref{prop:Mir} and Theorem \ref{theo:CNSHLPolytope} in order to get the conditions on the ages. Conversely, the equality of the ages shows that the hard Lefschetz property hold for $P$ (again by Theorem \ref{theo:CNSHLPolytope}) and by \cite[Lemma 5.1]{BCS} the preimage of $f_P$ under the mirror isomorphism $\varphi$ of Proposition \ref{prop:Mir} is the cohomology class 
of the $\qit$-ample divisor $\sum_{i=1}^r \ell_i^{-1} D_i$ where the positive integer $\ell_i$ is defined by $b_i =\ell_i a_i$, $a_i$ denoting the primitive lattice generator of the ray $\rho_i$ of the fan $\Sigma_P$.
\end{proof}

Fortunately, condition (\ref{eq:CNSAgeNonEntier}) has an easy combinatorial description. Let $\mathbf{\Sigma}_P =(\zit^n, \Sigma_P , \mathcal{V}(P))$ be the stacky fan of $P$.
For $\sigma$ a $n$-dimensional cone in the fan $\Sigma_P$, we denote  by
$\boite (\sigma )$ the set of the elements $v\in N$ such that $v=\sum_{\rho_i \subseteq\sigma} q_i b_i $ for some $0\leq q_i <1$ where $\rho_i$ is the ray generated by the vertex $b_i$ of $P$. Let $\boite (\mathbf{\Sigma}_P )$ be the union of $\boite (\sigma )$ for all $n$-dimensional cones $\sigma \in\Sigma_P$.

\begin{proposition} 
\label{prop:ConcretAge}
Condition (\ref{eq:CNSAgeNonEntier}) holds true
if and only if
$$[\nu (v) ]=(\dim \sigma (v) -1)/2\ \mbox{if}\ \nu (v)\notin\nit$$
and 
$$\nu (v)=\dim \sigma (v)/2 \ \mbox{if}\ \nu (v)\in\nit$$
for all $v\in\boite (\mathbf{\Sigma}_P )$, where $\sigma (v)$ the smallest cone of $\Sigma_P$ containing $v$ and $\nu$ is the Newton function  of $P$. 
\end{proposition}
\begin{proof} 
By \cite[Proposition 4.7]{BCS}, the sectors $\mathcal{X}_v$ are parametrized by $v\in\boite (\mathbf{\Sigma}_P )$ and $\dim \mathcal{X}_{v}=n-\dim \sigma (v)$. Let $v\in\boite (\mathbf{\Sigma}_P )$.
Because
$\dim \mathcal{X}_{v}=n-\age (\mathcal{X}_{v})-\age (\mathcal{X}_{v^{-1}})$, we get
$\age (\mathcal{X}_{v})+\age (\mathcal{X}_{v^{-1}})=\dim \sigma (v)$.
Therefore,
$[\age (\mathcal{X}_{v})]+[\age (\mathcal{X}_{v^{-1}})]=\dim \sigma (v)-1$
if $\age (\mathcal{X}_{v})\notin \nit$ 
and $[\age (\mathcal{X}_{v})]+[\age (\mathcal{X}_{v^{-1}})]=\dim \sigma (v)$
if $\age (\mathcal{X}_{v})\in \nit$. Because $\age (\mathcal{X}_v )=\nu (v)$ by \cite[Remark 5.4]{BCS}, the result follows from Theorem \ref{theo:CNSHLPolytope}.
\end{proof}

\noindent This is Theorem \ref{theo:Intro} in the introduction.

\section{Application to simplices}

We apply the previous results to simplices and we deduce some consequences in Hodge theory. 

\subsection{Hard Lefschetz property for simplices}
\label{sec:simplices}

In this text, we will say that the polytope $\Delta :=\conv (v_0 ,\cdots ,v_n )$ is a {\em simplex} if its vertices $v_i$ belong to the lattice $\zit^n$ and if it contains the origin as an interior point. 
The {\em weight} of a simplex $\Delta$ is
the tuple $Q(\Delta )=(q_0 ,\cdots , q_n )$ where 
\begin{equation}\nonumber 
q_i := |\det (v_0 ,\cdots , \widehat{v_i},\cdots , v_n )|
\end{equation}
for $i=0,\cdots ,n$. We will always assume that the tuple $Q(\Delta )=(q_0 ,\cdots , q_n )$ is arranged by increasing order (this can always be achieved by renumbering the vertices) and we will 
put $\mu :=q_0 +\cdots +q_n$.
The simplex $\Delta$ is {\em reduced} if $\gcd (q_0 ,\cdots ,q_n )=1$ (see \cite{Conrads}). 
Up to unimodular transformations, there exists a unique reduced simplex $\Delta$ of weight $(q_0 ,q_1 ,\cdots ,q_n )$ and an algorithm in order to construct it, and therefore to get $f_{\Delta}$, is given in \cite[Theorem 3.6]{Conrads} (recall that $f_{\Delta}$ denotes the Laurent polynomial defined by
$f_{\Delta} (u)=\sum_{i=0}^{n} u^{v_i}$
on $(\cit^* )^n$).

By Remark \ref{rem:Smooth}, the hard Lefschetz property is true if $(q_0 ,\cdots , q_n )=(1,\cdots , 1)$. We  now give a criterion about the remaining cases. Let $\Delta$ be a simplex of weight $(q_0,\cdots ,q_n )$. We define
\begin{equation}\nonumber
F:=\left\{\frac{\ell}{q_{i}}|\, 0\leq\ell\leq q_{i}-1,\ 0\leq i\leq n\right\}.
\end{equation}
We will denote by $f_{1},\cdots , f_{k}$ the elements of $F$ arranged by increasing order and we will put
\begin{equation}\nonumber
d_{i}:=\card \{j|\ q_{j}f_i \in\zit\}. 
\end{equation}
We have $f_1 =0$ and $d_1 =n+1$.

\begin{proposition}\label{prop:CNSHLPolSimplex}
Let $\Delta$ be a reduced simplex of weight $(q_0 ,\cdots , q_n )$ such that $q_n \geq 2$. 
Then $\Delta$ satisfies (HL) if and only if $f_{\Delta}$ satisfies (HL). 
And this happens if and only if
\begin{equation}\nonumber
 [-\mu f_i +\sum_{\ell =1}^{i-1} d_{\ell} ]=\frac{d_1 -d_i -1}{2}\ \mbox{for}\  i\geq 2
\end{equation}
if $-\mu f_i +\sum_{\ell =1}^{i-1} d_{\ell} \notin\zit$
and
\begin{equation}\nonumber
-\mu f_i  +\sum_{\ell =1}^{i-1} d_{\ell}=\frac{d_1 -d_i }{2}\ \mbox{for}\ i\geq 2
\end{equation}
if $-\mu f_i +\sum_{\ell =1}^{i-1} d_{\ell} \in\zit$.
\end{proposition}
\begin{proof} The first assertion follows from Corollary \ref{coro:CNSHLPol}. 
According to \cite[Section 3.4]{DoMa}, the sectors of $\mathcal{X}_{\Delta}$ are labelled by the set $F$ and
the ages of the sectors $\mathcal{X}_{f_{\ell}}$ are $\age (\mathcal{X}_{f_{1}})=0$ and
$\age (\mathcal{X}_{f_{i}})=\sum_{\ell =1}^{i-1} d_{\ell}-\mu f_{i}$
if $i=2,\cdots ,k$. By the proof of Proposition \ref{prop:ConcretAge}, (HL) holds if and only if $2[\age (\mathcal{X}_v )] =n-1-\dim \mathcal{X}_v$
if $\age (\mathcal{X}_v )\notin \nit$ and $2\age (\mathcal{X}_v ) =n-\dim \mathcal{X}_v$ if $\age (\mathcal{X}_v )\in \nit$. This gives the remaining assertions because $\dim \mathcal{X}_{f_{i}}=d_i -1$ and $d_1 =n+1$.
\end{proof}

\begin{example}\label{example:HLSimplex}
We give here two basic examples.
\begin{enumerate}
\item Let $\Delta$ be the reduced simplex of weight $Q(\Delta )= (1,1,3)$. Then $f_1 =0$, $f_2 =1/3$, $f_3 =2/3$, $d_1 =3$,
$d_2 =1$, $d_3=1$ and $\mu =5$: the simplex $\Delta$ does not satisfy (HL).
\item Let $\Delta$ be the reduced simplex of weight $Q(\Delta )= (1,2,2,3)$, for which $\mu =8$. Then $f_1 =0$, $f_2 =1/3$,
$f_3 =1/2$, $f_4 =2/3$, $d_1 =4$, $d_2 =1$, $d_3 =2$, $d_4 =1$: the simplex $\Delta$ satisfies (HL).
\end{enumerate}
\end{example}

Recall that the reduced simplex $\Delta$ of weight $(q_0 ,\cdots , q_n)$ is reflexive if and only if $q_i$ divides $\mu$ for $i=0,\cdots ,n$, see \cite[Proposition 5.1]{Conrads}.

\begin{corollary}\label{coro:NCSHLSimplices}
A reduced and reflexive simplex $\Delta$ of weight $(q_0 ,\cdots , q_n )$ with $q_n \geq 2$ satisfies (HL) if and only if
$$-\mu f_i +\sum_{\ell =1}^{i-1} d_{\ell}=(d_1 -d_i )/2$$
for $i=2,\cdots ,k$. \qed
\end{corollary}

\begin{remark}\label{rem:NCHLSimplices}
Assume that the reduced and reflexive simplex $\Delta$ satisfies (HL). Then, if $q_n \geq 2$, it follows from Corollary \ref{coro:NCSHLSimplices} that we must have
\begin{equation}\label{eq:NCHLSimplices}
\frac{2\mu }{q_n }=n+1 +m(q_n ) 
\end{equation}
where $m(q_n )$ denotes the multiplicity of $q_n$ in the tuple $(q_0 ,\cdots , q_n )$ because $f_2 =1/q_n$ and $d_2 =m(q_n )$ (recall that we assume that the tuple $(q_0 ,\cdots , q_n )$ is arranged by increasing order): this is Proposition \ref{prop:NCHLSimplicesIntro} in the introduction. Most of the time it will be enough to notice that this necessary condition does not hold in order to show that the hard Lefschetz condition (HL) fails for $\Delta$. 
\end{remark}

\begin{example} \label{ex:HLSimplicesSmallDim}
Reduced and reflexive simplices are classified up to dimension four in \cite{Conrads}. Using Corollary \ref{coro:NCSHLSimplices} and Remark \ref{rem:NCHLSimplices}, we get the following statements:
\begin{itemize}
\item two dimensional reduced and reflexive simplices satisfy the hard Lefschetz property;
\item if $n=3$,  there are $14$ reduced and reflexive simplices (up to unimodular transformations) and
the hard Lefschetz property hold only for the simplices with weights $(1,1,1,1)$ and $(1,1,2,2)$.
\item if $n=4$, there are $147$ reduced and reflexive simplices (up to unimodular transformations) and the hard Lefschetz property hold only for the simplices $\Delta$ with weights $(1,1,1,1,1)$, $(1,1,1,1,2)$, $(1,1,2,2,2)$, $(1,2,3,3,3)$ and $(1,2,2,3,4)$.
\end{itemize}
Of course, our results apply to greater dimensions: for instance, it is immediately seen the reduced and reflexive simplex of weight $(1,1,1,1,1,1,3)$ in $\rit^6$ does not satisfy the hard Lefschetz property.
\end{example}

\subsection{Application to Hodge theory for reflexive simplices}
\label{sec:KKP}

We keep in this section the setting and the notations of \cite{Sab}.
Let $f$ be a convenient and nondegenerate Laurent polynomial on $(\cit^*)^n$ and let $P$ be its Newton polytope.
It is known that $f$ defines a mixed Hodge structure $MHS_f :=(H, F^{\bullet}H, W_{\bullet}H)$ and
this mixed Hodge structure is said to be of {\em Hodge-Tate type} if
\begin{enumerate}
\item $W_{2i+1}H=W_{2i}H$ for $i\in\zit$,
\item the filtrations $F^{\bullet} H$ and $W_{2\bullet}$ are opposite, that is $\gr_F^p \gr_{2q}^W H =0$ for $p\neq q$.
\end{enumerate}
\noindent The link with the hard Lefschetz property is given by the following result:

\begin{proposition}\cite[Corollary 2.6]{Sab}
\label{prop:CNSHodgeTatef}
The following are equivalent:
\begin{enumerate}
\item the mixed Hodge structure $MHS_f$ is of Hodge-Tate type,
\item  $f$ satisfies the hard Lefschetz property of Definition \ref{def:HLforfGen}.\qed
\end{enumerate}
\end{proposition}

When $P$ is reflexive, we will say that $f$ {\em satisfies the KKP conjecture} if
$\dim\gr^p_F H =\dim \gr^W_{2p} H$ (see \cite[Conjecture 3.6]{KKP}, but also \cite[3.a]{Sab} and \cite{Sha}).
We keep the notations of Section \ref{sec:simplices}.

\begin{proposition}\label{prop:KKPSimplices}
Let $\Delta$ be a reduced and reflexive simplex in $\rit^n$ with weight $Q(\Delta )=(q_0 ,\cdots , q_n )$, where $q_n \geq 2$. 
The Laurent polynomial $f_{\Delta}$ satisfies the KKP conjecture if and only if
\begin{equation}\nonumber 
\sum_{\ell =1}^{i-1} d_{\ell}-\mu f_i =\frac{d_1 -d_i }{2}
\end{equation}
for $i=2,\cdots ,k$. 
\end{proposition}
\begin{proof} 
By \cite[Lemma 2.4 and Corollary 2.6]{Sab}, $f$ satisfies the KKP conjecture if and only if $f$ satisfies (HL). Thus, the result
follows from Proposition \ref{prop:CNSHLPolSimplex} and Corollary \ref{coro:NCSHLSimplices}. 
\end{proof}

\noindent If $\Delta$ is a reduced and reflexive simplex in $\rit^n$ for $n=2,3$, we get from Example \ref{ex:HLSimplicesSmallDim} the Laurent polynomials $f_{\Delta}$ which satisfy the KKP conjecture.

\section{Application to the distribution of spectral numbers}
\label{sec:Distribution}

We apply the previous results to the study of the distribution of the spectrum at infinity of a convenient and nondegenerate Laurent polynomial $f$ defined on $(\cit^* )^n$. Recall that a polynomial $a_0 +a_1 z+\cdots +a_n z^n$ is unimodal if there exists an index $j$ such that $a_i \leq a_{i+1}$ for all $i<j$ and $a_i \geq a_{i+1}$ for all $i\geq j$.

\subsection{Unimodality of the spectrum at infinity: unipotent case}
\label{sec:UnimodalPol}

We study in this section the unimodality of the spectrum at infinity if $f$
satisfies the assumption of Corollary \ref{coro:UnipotentReflexif}.
So let us assume that 
$$\Spec_f (z) =1+d(1)z+\cdots +d(n-1)z^{n-1} +z^n$$
where $d(i ):=\dim_{\cit} \gr_{i} ^{\mathcal{N}}\mathcal{A}_f$ for $i=1,\cdots ,n-1$.
The results in this section follow from well-known in combinatorics.
The first one is due to Hibi \cite{Hibi0}:

\begin{proposition}\label{prop:Unimn5}
We have $1\leq d(1)\leq d(i)\ \mbox{for}\  i\leq [n/2]$.
In particular, $\Spec_f (z)$ is unimodal if $n\leq 5$.
\end{proposition}
\begin{proof} Let $P$ be the Newton polytope of $f$ and let $\delta_P (z) =\delta_0 +\delta_1 z +\cdots +\delta_n z^n$ be its $\delta$-vector. 
By Corollary \ref{coro:UnipotentReflexif} and \cite[Corollary 2.2]{D12}, $P$ is reflexive and $\Spec_f (z) =\delta_0 +\delta_1 z +\cdots +\delta_n z^n$.
By \cite{Hibi0},  we have $\delta_0 \leq \delta_1\leq \delta_j$ for $2\leq j\leq [n/2]$.
The inequalities follow and we use then the symmetry $d(i)=d(n-i)$ in order to get 
the unimodality for $n\leq 5$.
\end{proof}

\noindent Nevertheless, in this situation the spectrum at infinity needs not to be unimodal if $n\geq 6$.
The following counter-example is provided by \cite{Payne}:

\begin{proposition} 
\label{prop:ContreExfUnimodal}
Let $s\geq 2$, $k\geq 2$ be two integers and let $n:=sk$.
Let $f_{\Delta}$ 
be the Laurent polynomial defined by 
$$f_{\Delta} (u_1 ,\cdots , u_n ):=u_1 +\cdots +u_n +\frac{1}{u_1 \cdots u_{n-1} u_n^s}$$
on $(\cit^*)^n$. Then,
\begin{enumerate}
\item $f_{\Delta}$ is convenient and nondegenerate,
\item the Milnor number of $f_{\Delta}$ is equal to $s(k+1)$,
\item $\Spec_{f_{\Delta}}(z)=1 + z + \cdots +z^{sk} + z^{(s-1)k}+z^{(s-2)k}+\cdots +z^k $,
\item the spectrum at infinity of $f_{\Delta}$ is unimodal if and only if $s=2$,
\item  $f_{\Delta}$ satisfies the hard Lefschetz property if and only if $s=2$.
\end{enumerate}
\end{proposition}
\begin{proof} 
Let $\Delta :=\conv (e_1 ,\cdots ,e_n ,-\sum_{i=1}^n q_i e_i )$
where $n:=sk$, $(e_1 ,\cdots ,e_n )$ is the canonical basis of $\rit^n$ and $(q_1 ,\cdots ,q_n ):=(1,\cdots ,1,s)$ where $1$ is counted $sk-1$-times.  
The simplex $\Delta$ is reduced and reflexive and is
the Newton polytope of $f_{\Delta}$. 
Its weight is $(q_0 ,q_1 ,\cdots ,q_n )=(1,\cdots ,1,s)$ where $1$ is counted $sk$-times and $\mu_{\Delta}=s(k+1)$.
The nondegeneracy follows from the fact that the facets of $\Delta$ are simplices. The assertion on the Milnor number follows from \cite{K}. Using the results recalled in Section \ref{sec:simplices}, we get
$f_1 =0, f_2 =1/s,\cdots , f_s =(s-1)/s$, $d_1=n+1, d_2 =\cdots =d_s =1$. Define $\beta_1 :=0$ and 
$$\beta_i :=d_1 +\cdots +d_{i-1}-\mu f_i =k(s-(i-1))$$ 
for $i=2, \cdots ,s$. By \cite{DoMa},
the spectrum at infinity of $f_{\Delta}$ is given by
$\beta_{1} , \beta_{1}+1,\cdots , \beta_{1}+d_{1}-1,  
\cdots ,
\beta_{k} , \beta_{k}+1,\cdots , \beta_{k}+d_{k}-1$, and the formula for $\Spec_{f_{\Delta}}(z)$ follows. The assertion about unimodality is clear and for the last statement, notice that the necessary and sufficient condition of Corollary \ref{coro:NCSHLSimplices} is $s=2(i-1)$ for $i=2,\cdots ,s$ and is satisfied only for $s=2$.
\end{proof}

\begin{remark}
Because $\Delta$ is reflexive, $\Spec_{f_{\Delta}}(z)$ is equal to the $\delta$-vector of $\Delta$, see Proposition \ref{prop:ReflexivevsIntegral}. This formula 
for the  $\delta$-vector of $\Delta$ can already be found in \cite{Payne}. 
\end{remark}

If $n\geq 6$, we have the following positive result:

\begin{proposition}
\label{prop:HLUnimRef}
Let $f$ be a convenient and nondegenerate Laurent polynomial on $(\cit^* )^n$ whose spectrum at infinity is a polynomial. Assume that $f$ satisfies the hard Lefschetz property of definition \ref{def:HLforfGen}.
Then $\Spec_f(z)$ is unimodal.
\end{proposition}
\begin{proof} 
The hard Lefschetz property shows that $[f]:\gr_{i-1} ^{\mathcal{N}}\mathcal{A}_f \longrightarrow \gr_{i} ^{\mathcal{N}}\mathcal{A}_f $ is injective for $i\leq n/2$ and surjective for $i>n/2$. 
\end{proof}

\noindent Of course, the converse is not true: by Proposition \ref{prop:Unimn5}, $\Spec_{f_{\Delta}}(z)$ is unimodal for any four dimensional reduced and reflexive simplex $\Delta$ but $f_{\Delta}$ does not satisfy the hard Lefschetz property in general (see Example \ref{ex:HLSimplicesSmallDim}).

\subsection{Unimodality of the spectrum at infinity: the general case}

We consider now the general case, that is when the spectrum at infinity of $f$ is not necessarily a polynomial:
we write 
$$\Spec_f (z) =\sum_i d(\alpha_i )z^{\alpha_i}$$ 
where $d(\alpha_i ):=\dim_{\cit} \gr_{\alpha_i} ^{\mathcal{N}}\mathcal{A}_f$ and $\alpha_i \in\qit$, the rational numbers $\alpha_i$ being arranged by increasing order.
Using the symmetry property
 $z^n \Spec_f (z^{-1})=\Spec_f (z)$, one would expect that
\begin{equation}\label{eq:UnimSpectre} 
d(\alpha_1 )\leq d(\alpha_2 )\leq \cdots \leq  d(\alpha_{\ell})
\end{equation}
for all $\alpha_{\ell} \leq n/2$.
Unfortunately, and unlike Section \ref{sec:UnimodalPol}, this may fail if $n\leq 5$ or if $f$ satisfies the hard Lefschetz property (HL) (see example \ref{ex:ContreExUnimGen} below).

So what gives in this case this hard Lefschetz property?
Let us write $\Spec_f (z) = \sum_{\alpha\in [0,1[} z^{\alpha} \Spec_f^{\alpha} (z)$
where $\Spec_f^{\alpha} (z)\in\qit [z]$.

\begin{proposition}
\label{prop:HLUnimGen}
Assume that $f$ satisfies the hard Lefschetz property of definition \ref{def:HLforfGen}.
Then the polynomials $\Spec_f^{\alpha} (z)$ are unimodal for $\alpha\in [0,1[$.
\end{proposition}
\begin{proof} 
The hard Lefschetz assumption shows that $[f]:\gr_{\alpha +i-1} ^{\mathcal{N}}\mathcal{A}_f \longrightarrow \gr_{\alpha +i} ^{\mathcal{N}}\mathcal{A}_f $ is injective for $i\leq (n-1)/2$ and surjective for $i>(n-1)/2$ for $\alpha \in ]0,1[$. The case $\alpha =0$ has been considered in Proposition \ref{prop:HLUnimRef}.
\end{proof}

\begin{example} \label{ex:ContreExUnimGen}
Let $f$ 
be the Laurent polynomial defined by 
$f(u_1 , u_2 , u_3 )=u_1 +u_2 +u_3 +1/u_1^2 u_2^2 u_3^3 $
on $(\cit^*)^3$. Then
$\Spec_{f}(z)=1 + 2z +z^{4/3}+z^{5/3}+ 2z^{2}+z^3$ 
and does not satisfy (\ref{eq:UnimSpectre}).
However, $f$ satisfies (HL) (see Example \ref{example:HLSimplex}). We have $\Spec_f^{0} (z)=1+2z+2z^2 +z^3$, $\Spec_f^{1/3} (z)=z$,  $\Spec_f^{2/3} (z)=z$ and these polynomials are unimodal. 
\end{example}

\end{document}